\documentclass[a4paper,10pt]{amsart}
\usepackage[english]{babel}
\usepackage[T1]{fontenc}
\usepackage{graphicx}
\usepackage{tikz, pgfplots}
\usepackage{hyperref}
\usepackage{amssymb}
\usepackage{amsmath}
\usepackage{amsfonts}
\usepackage{amsthm}
\usepackage{amscd}
\usepackage{mathrsfs}
\usepackage{tikz, pgfplots}
\usepackage{lineno}

\theoremstyle{plain}
\newtheorem{thm}{Theorem}[section]
\newtheorem{prop}[thm]{Proposition}

\newtheorem{cor}[thm]{Corollary}
\newtheorem{rem}[thm]{Remark}

\newtheorem{exm}[thm]{Example}

\theoremstyle{definition}
\newtheorem{defin}[thm]{Definition}

\newcommand{\K}{\mathbb{K}}

\newcommand{\p}{\mathbb{P}}

\newcommand{\E}{\mathcal{E}}

\newcommand{\NN}{{\mathcal N}}

\DeclareMathOperator{\indeg}{indeg}
\DeclareMathOperator{\defect}{def}
\DeclareMathOperator{\Syz}{Syz}

\title{On the Hilbert vector of the Jacobian module of a plane curve}

\author{Armando Cerminara*}
\address{*Dipartimento Matematica ed Applicazioni ``Renato Caccioppoli", Universit\`a Degli Studi Di Napoli ``Federico II", Via Cinthia - Complesso Universitario Di Monte S. Angelo 80126 - Napoli - Italia}
\email{armando.cerminara@unina.it, giovanna.ilardi@unina.it}

\author[Alexandru Dimca]{Alexandru Dimca$^{1}$}
\address{Universit\'e C\^ ote d'Azur, CNRS, LJAD and INRIA, France}
\email{dimca@unice.fr}

\author{Giovanna Ilardi*}

\thanks{$^1$ This work has been partially supported by the French government, through the $\rm UCA^{\rm JEDI}$ Investments in the Future project managed by the National Research Agency (ANR) with the reference number ANR-15-IDEX-01}


\begin{document}
\maketitle

\begin{abstract}
We identify several classes of complex projective plane curves $C:f=0$, for which the Hilbert vector of the Jacobian module $N(f)$ can be completely determined, namely the 3-syzygy curves, the maximal Tjurina curves and the nodal curves, having only rational irreducible components. A result due to  Hartshorne, on the cohomology of some  rank 2 vector bundles on $\p^2$, is used to get a sharp lower bound for the initial degree of the Jacobian module $N(f)$, under a semistability condition.

   \end{abstract}

\section*{Introduction}
Let $S=\mathbb{C}[x,y,z]$ be the graded polynomial ring in three variables $x,y,z$ with complex coefficients. Let $C: f=0$ be a reduced curve of degree $d$ in the complex projective plane $\p^2$. We denote by $J_f=(f_x,f_y,f_z)$ the Jacobian ideal, i.e. the homogeneous ideal in $S$ spanned by the partial derivatives $f_x,f_y,f_z$ of $f.$ Since $C$ is reduced, the singular subscheme of $C$, which is defined by the Jacobian ideal $J_f$, is 0-dimensional, and its degree is denoted by $\tau(C)$, and is called the global Tjurina number of $C$.
Consider the graded $S-$module of Jacobian syzygies of $f$, namely $$\Syz(J_f)=\{(a,b,c)\in S^3 : af_x+bf_y+cf_z=0\}.$$
Let $mdr(f):=\min\{k : \Syz(J_f)_k\neq (0)\}$ be the minimal degree of a Jacobian syzygy for $f$; in this paper we will assume $mdr(f)\geq 1$ unless otherwise specified. In fact, if $mdr(f)=0$, the curve $C$ is a pencil of lines.\\
We say that  $C: f=0$ is a $m-$syzygy curve if the $S-$module $\Syz(J_f)$ is minimally generated by $m$ homogeneous syzygies, $r_1,r_2,\ldots,r_m$, of degree $d_i=\deg r_i$, ordered such that $$1\leq d_1\leq d_2\leq\ldots\leq d_m.$$ The multiset $(d_1,d_2,\ldots,d_m)$ is called the exponents of the plane curve $C$ and $\{r_1,r_2,\ldots,r_m\}$ is said to be a minimal set of generators for the $S-$module $\Syz(J_f).$ Some of the $m-$syzygy curves have been carefully studied.
We recall that:
\begin{itemize}
    \item a $2-$syzygy curve $C$ is said to be {\emph{free}},  since then the $S-$module $\Syz(J_f)$ is a free module of the rank $2$, see \cite{BGLH17, D17, DS17, S05, S06, ST14};
    \item a $3-$syzygy curve is said to be {\emph{nearly free}} when $d_3=d_2$ and $d_1+d_2=d$, see \cite{DS18, AD18, BGLH17, D17, DArX, MVArX};
    \item a $3-$syzygy line arrangement is said to be a {\emph{plus-one generated line arrangement}} of level $d_3$ when $d_1+d_2=d$ and $d_3 \geq d_2$, see \cite{AArX}. By extension, a $3-$syzygy curve $C$ is said to be a {\emph{plus-one generated curve}} of level $d_3$  when $d_1+d_2=d$ and $d_3 \geq d_2$, see \cite{DSIn}.
\end{itemize}

\par The Jacobian module of $f$, or of the plane curve $C:f=0$, is the quotient module $N(f)=\hat{J_f}/J_f$, with $\hat{J_f}$ the saturation of the ideal $J_f$ with respect to the maximal ideal $\mathbf{m}=(x,y,z)$ in $S$. The Jacobian module $N(f)$ coincides with $H^0_\mathbf{m}(S/J_f)$, see \cite{Se14}. Let $n(f)_j=\dim N(f)_j$,  $T=3(d-2)$ and recall that the Jacobian module $N(f)$ enjoys a weak Lefschetz type property, see \cite{DP16} for this result, and \cite{HMNW03, HMMNWW13, I18} for Lefschetz properties of Artinian algebras in general. More precisely, we have 
\begin{equation}\label{LPN} n(f)_0\leq n(f)_1\leq\ldots\leq n(f)_{\lfloor\frac{T}{2}\rfloor-1}\leq n(f)_{\lfloor\frac{T}{2}\rfloor}\geq n(f)_{\lfloor\frac{T}{2}\rfloor+1}\geq\ldots\geq n(f)_T.\end{equation}

 We consider the following two invariants for a curve $C:f=0$  \begin{eqnarray*}\sigma(C):=\min\left\{j : n(f)_j\neq 0\right\}=\indeg(N(f)), & \nu(C):=\max\left\{n(f)_j\right\}_j.\end{eqnarray*} The self duality of the graded $S-$module $N(f)$, see \cite{HS12, Se14, SW15}, implies that 
\begin{equation}\label{sym}  
 n(f)_j=n(f)_{T-j}, 
\end{equation}  
for any integer $j$, in particular $n(f)_k\neq 0$ exactly for $k=\sigma(C),\ldots,T-\sigma(C).$\\ 

The main aim of this paper is to identify classes of curves $C:f=0$ for which the Hilbert vector $(n(f)_j)$ of the Jacobian module $N(f)$ can be completely determined. In \cite[Theorem 3.1, Theorem 3.2]{DSInB}, recalled below in Theorem \ref{THMD}, there is a description of the dimensions
$n(f)_j$
for a certain range of $j$. Moreover, in \cite[Theorem 3.9, Corollary 3.10]{DSIn}, recalled below in Theorem \ref{THM10} and Corollary \ref{COR10}, there are descriptions of the minimal resolution of $N(f)$, when $C:f=0$ is a $3-$syzygy curve, and respectively a plus-one generated curve of degree $d\geq 3$. Using these results, we first give a general formula for the Hilbert vector $(n(f)_j)$ of the Jacobian module of a $3-$syzygy curve in Theorem \ref{thmmain1}, as well as a graphic representation of its behavior.
Then we determine the Hilbert vector $(n(f)_j)$ of the Jacobian module $N(f)$ when $C:f=0$ is a maximal Tjurina curve, see Proposition \ref{propMT}. Next we get some information on the Hilbert vector $(n(f)_j)$  when $C:f=0$ is a nodal curve, which is complete if in addition all the irreducible components of $C$ are rational, see Theorem \ref{thmNodal}.

In the final section we  use a result due to  Hartshorne, see \cite[Theorem 7.4]{Hart}, to relate the cohomology of some rank 2 vector bundles on $\p^2$ to the
Hilbert vector $(n(f)_j)$ of the Jacobian module $N(f)$. More precisely,
we get in this way a sharp lower bound for the initial degree $\sigma(C)$ of the Jacobian module $N(f)$, under the condition $mdr(f) \geq (d-1)/2$, see Theorem
\ref{thmH1}.

We would like to thank the referee for his careful reading of our manuscript and for his many very useful suggestions to improve the presentation.

\section{Preliminaries}

We recall some notations and results. Let $C:f=0$ be a reduced complex plane curve of $\p^2$, assumed not free, and consider the Milnor algebra $M(f)=S/J_f$, where $J_f=(f_x,f_y,f_z)$ is the Jacobian ideal.
The general form of the minimal resolution for the Milnor algebra $M(f)$ of such a curve $C:f=0$ 
is 
\begin{equation}\label{RMM}0\to \displaystyle{\oplus_{i=1}^{m-2}S(-e_i)}\to \displaystyle{\oplus_{i=1}^{m}S(1-d-d_i)}\to S^3(1-d)\to S,\end{equation} 
with $e_1\leq e_2\leq\ldots\leq e_{m-2}$ and $1 \leq d_1\leq d_2\leq\cdots\leq d_m$. It follows from \cite[Lemma 1.1]{HS12} that one has $$e_j=d+d_{j+2}-1+\epsilon_j,$$ for $j=1,\ldots, m-2$ and some integers $\epsilon_j\geq 1.$
The minimal resolution of $N(f)$ obtained from (\ref{RMM}), by \cite[Proposition 1.3]{HS12}, is
 $$0\to \displaystyle{\oplus_{i=1}^{m-2}S(-e_i)}\to\displaystyle{\oplus_{i=1}^mS(-\ell_i)}\to\displaystyle{\oplus_{i=1}^m}S(d_i-2(d-1))\to \displaystyle{\oplus_{i=1}^{m-2}S(e_i-3(d-1))},$$
 where $\ell_i=d+d_i-1$. It follows that 
\begin{equation}\label{sigma} 
 \sigma(C)=3(d-1)-e_{m-2}=2(d-1)-d_m-\epsilon_{m-2}.
\end{equation} 
The following result describes the central part of the Hilbert vector of $N(f)$.
\begin{thm}\label{THMD}
Let $C:f=0$ be a reduced, non free curve of degree $d$ and set $r=mdr(f)$. Then one has the following.
\begin{enumerate}
    \item[(i)]if $r\geq \frac{d}{2}$ and $2d-4-r\leq j \leq d-2+r$, then 
    \begin{equation}\label{EQTHMD}n(f)_j=\begin{cases}
     3(d')^2-(j-3d'+2)(j-3d'+1)-\tau(C) \mbox{ for } d=2d'+1\\
     3(d')^2-3d'+1-(j-3d'+3)^2-\tau(C)  \mbox{ for } d=2d'
    \end{cases}\end{equation}
    \item[(ii)] if $r<\frac{d}{2}$ and $d+r-3\leq j\leq 2d-r-3$,
    then $n(f)_j=\nu(C)$. Moreover $n(f)_{d+r-4}=n(f)_{2d-r-2}=\nu(C) -1$.
\end{enumerate}
\end{thm}
\begin{proof}
 See \cite[Theorem 3.1 and Theorem 3.2]{DSInB}.
\end{proof}
 By Theorem \ref{THMD}, in case (i), the points $(j,n(f)_j)$ lie on an upward pointing parabola. Moreover, using the formulas (\ref{cond3}) and (\ref{cond4}) and Remark \ref{rk4.1}, the claim (\ref{EQTHMD}) can be written: 
 $$n(f)_j=\begin{cases}
     \nu(C)-(j-\lfloor\frac{T}{2}\rfloor)(j-\lceil\frac{T}{2}\rceil) \mbox{ for } d=2d'+1\\
     \nu(C)-(j-\frac{T}{2})^2  \mbox{ for } d=2d'\end{cases}$$ with $T=3(d-2)$ as above.
     On the other hand, in the case (ii), the points $(j, n(f)_j)$ lie on a horizontal line segment, with a one-unit drop at the extremities, as represented in Figure 1 below.
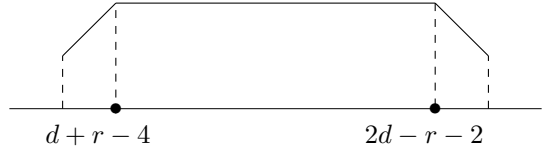
\begin{figure}[h]
\centering
\begin{tikzpicture}[scale=0.70]
\draw[black, ultra thin] (0,0) -- (10,0);
\draw[black, thin] (1,1) -- (2,2);
\draw[black, thin] (2,2) -- (8,2);
\draw[black,thin] (8,2)-- (9,1);
\node[right] at (0.5,-0.5) {$d+r-4$};
\draw[dashed, thin] (1,0) -- (1,1);
\draw[dashed, thin] (2,-0.09)--(2,2);
\draw[dashed,thin] (9,0)--(9,1);
\draw[dashed,thin] (8,2)--(8,-0.1);
\node[right] at (6.5,-0.5) {$2d-r-2$};
 \node at (2,0) {$\bullet$};
\node at (8,0) {$\bullet$}; 
\end{tikzpicture}


\caption{The case $r < {d \over 2}$.}
\label{fig: case 1}
\end{figure}

Recall the following definition, see \cite{D17, DS17}.

\begin{defin}
\label{def}
For a plane curve $C:f=0$, the {\it coincidence threshold} of $f$ is the integer
$$ct(f)=\max \{q:\dim M(f)_k=\dim M(f_s)_k \text{ for all } k \leq q\},$$
with $f_s$  a homogeneous polynomial in $S$ of the same degree $d$ as $f$ and such that $C_s:f_s=0$ is a smooth curve in $\p^2$.

\end{defin}
It is known that $ct(f) \geq d-2+mdr(f)$.\\
Note that $C: f=0$ is a free curve if and only if $\nu(C)=0$, hence $N(f)=0$. In other words, the Jacobian ideal is satured in this case, i.e. $\hat{J}_f=J_f$, see \cite{DS14,ST14}. The first nontrivial case is that of nearly free curves; indeed, by \cite[Corollary 2.17]{DS18}, for a nearly free curve $C:f=0$, one has $N(f)\neq 0$ and $\nu(C)=1$. Moreover $\sigma(C)=d+d_1-3$ and this describes completely the Hilbert vector of the Jacobian module of a nearly free curve. In particular it has the  shape described on the left side of Figure 2.
\begin{figure}[h]
\centering
\begin{tikzpicture}[scale=0.30]
\draw[black, ultra thin] (0,0) -- (10,0);
\draw[dashed, thin] (2,-0.09) -- (2,1);
\node[right] at (1,-1) {$\sigma(C)$};
\draw[black, thin] (2,1) -- (8,1);
\node[right] at (7.5,-1) {$T-\sigma(C)$};
\draw[dashed, thin] (8,1) -- (8,-0.09);
\node[right] at (4.5,-1) {$\frac{T}{2}$};
\draw[dashed, thin] (5,-0.09) -- (5,2);
\end{tikzpicture}
\hspace*{0.2in}
\begin{tikzpicture}[scale=0.30]
\draw[black, ultra thin] (0,0) -- (20,0);
\draw[black, thin] (2,1) -- (8,7);
\draw[dashed, thin] (2,-0.09) -- (2,1);
\draw[dashed, thin] (2,1) -- (8,1);
\draw (3,1) arc (0:45:1);
\node[right] at (1,-1) {$\sigma(C)=k_3$};
\node[right] at (3,1.5) {$\frac{\pi}{4}$};
\draw[dashed, thin] (8,-0.09) -- (8,7);
\node[right] at (7.5,-1) {$k_2$};
\draw[black, thin] (8,7) -- (16,7);
\draw[dashed, thin] (12,7) -- (12,-0.09);
\node[right] at (11.5,-1) {$\frac{T}{2}$};
\draw[decorate, decoration={brace,mirror}] (14,0.5) -- (14,6.5);
\node[right] at (15,3.5) {$d_3-d_2+1=\nu(C)$};
\end{tikzpicture}
\caption{The case of plus-one generated curves}
\label{fig:3syz}
\end{figure}
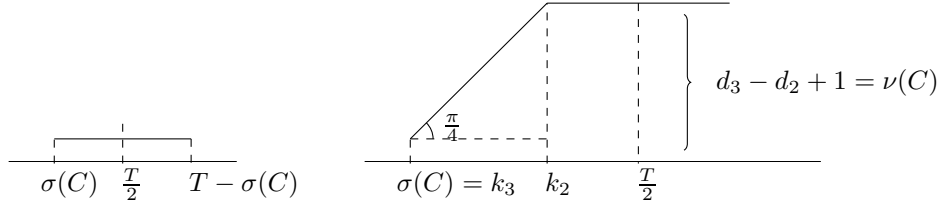
Recall that a nearly free curve is exactly a plus-one generated curve with exponents satisfying $d_2=d_3$.
For the more general case of the $3-$syzygy curves, we recall the following result.
\begin{thm}\label{THM10}\cite[Theorem 3.9]{DSIn} Let $C:f=0$ be a $3-$syzygy curve with exponents $(d_1,d_2,d_3)$ and set $e=d_1+d_2+d_3$. Then the minimal free resolution of $N(f)$ as a graded $S-$module has the form 
    \begin{equation}\label{MRNF}0\to S(-e)\to \displaystyle{\oplus_{i=1}^3S(1-d-d_i)}\to \displaystyle{\oplus_{i=1}^3S(d_i+2-2d)}\to S(e+3-3d),\end{equation} where the leftmost map is the same as in the resolution (\ref{RMM}), when $m=3$. In particular, $$\sigma(C)=3(d-1)-(d_1+d_2+d_3).$$

\end{thm}
This implies the following.
\begin{cor}\label{COR10} \cite[Corollary 3.10]{DSIn}
 Let $C:f=0$ be a plus-one generated curve of degree $d\geq 3$ with $(d_1,d_2,d_3)$, which is not nearly free, i.e. $d_2<d_3.$ Set $k_j=2d-d_j-3$ for $j=1,2,3$. Then one has the following minimal free resolution of $N(f)$ as a graded $S-$module:
 \begin{align*}0\to S(-d-d_3)\to S(-d-d_3+1)\oplus S(-k_1-2)\oplus S(-k_2-2)\to\\ \to S(-k_1-1)\oplus S(-k_2-1)\oplus S(-k_3-1)\to S(-k_3).\end{align*} In particular $\sigma(C)=k_3<k_2 \leq \frac{T}{2}$ and the Hilbert vector of $N(f)$ is given by following formulas: 
 \begin{enumerate}
     \item $n(f)_j=0$ for $j<k_3$;
     \item $n(f)_j=j-k_3+1$ for $k_3\leq j\leq k_2$;
     \item $n(f)_j=d_3-d_2+1=\nu(C)$ for $k_2\leq j\leq \frac{T}{2}.$
 \end{enumerate}
\end{cor}
By above corollary, the Hilbert vector of the Jacobian module of a plus-one generated curve of degree $d$ and level $d_3$ has the shape given on the right hand side of Figure 2, {\it where we have drawn only the part corresponding to $j \leq \frac{T}{2}$, due to the symmetry} \eqref{sym}.

 \section{Results on the Hilbert vector of $N(f)$ for 3-syzygy curves}

As a simple example of a 3-syzygy curve which is not a plus-one generated curve, let $C:f=0$ be a smooth curve of degree $d\geq 3$,  where $d_1=d_2=d_3=d-1$. It is known that the Hilbert function of the Milnor algebra $M(f)$ is in this case $\left(\frac{1-t^{d-1}}{1-t}\right)^3$. For a smooth curve we have $N(f)=M(f)$, hence $n(f)_j=\dim M(f)_j$ and the Hilbert vector of the Jacobian module $N(f)$ has the shape described in Figure 3. It is interesting to notice the change in convexity when we pass through the value $j=d-1$.    
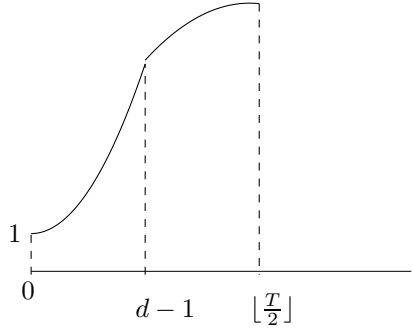
\begin{figure}[h]
\centering
\begin{tikzpicture}[scale=0.50]
\draw[black, ultra thin] (0,0) -- (10,0);
\draw[black, thin] plot[smooth, domain=0:3] (\x, {0.5*(\x)^2+1} );
\node[right] at (2.5,-1) {$d-1$}; 
\draw[dashed, thin] (3,-0.09)--(3,0.5*3^2+1);
\node[left] at (0,1) {$1$};
\node[right] at (-0.5,-0.5) {$0$};
\draw[dashed, thin] (0,-0.09)--(0,1);
\draw[black, thin] plot[domain=3:6] (\x, {-0.2*\x^2+2.3*\x+0.5});
\draw[dashed, thin] (6,7.1)--(6,-0.09);
\node[right] at (5.5,-1) {$\lfloor\frac{T}{2}\rfloor$};
\end{tikzpicture}
\caption{The case of smooth curves}
\label{fig:smooth}
\end{figure}

For a general $3-$syzygy curve, we have the following result.

\begin{thm} \label{thmmain1}
Let $C:f=0$ be a $3-$syzygy curve of degree $d$, not plus-one generated, with exponents $d_1\leq d_2\leq d_3$. Set $e=d_1+d_2+d_3$ and $k_i=2(d-1)-d_i$ for $i=1,2,3$. Then the following hold.
$$n(f)_k=\begin{cases}
 0 \mbox{ for } k< \sigma\\
 \binom{k-\sigma+2}{2} \mbox{ for } \sigma \leq k < k_3\\
 \binom{k-\sigma+2}{2}-\binom{k-k_3+2}{2} \mbox{ for } k_3\leq k<k_2\\
 \binom{k-\sigma+2}{2}-\binom{k-k_3+2}{2}-\binom{k-k_2+2}{2} \mbox{ for } k_2\leq k< T_0,

\end{cases}$$
where $\sigma=\sigma(C)=3(d-1)-e$ and $$T_0=\begin{cases}k_1 -1\mbox{ if } d_1\geq \frac{d}{2}\\ d+d_1-2 \mbox{ if } d_1<\frac{d}{2}. \end{cases}$$
\end{thm}
Note that  $n(f)_k$ is known  for  $T_0\leq k\leq \frac{T}{2} $ in view of Theorem \ref{THMD}, hence the information on the Hilbert vector of $N(f)$ is complete in this situation.
\begin{proof}
Note that $\sigma\geq 0$, since, by  \cite[Theorem 2.4]{DSIn} we have $d_j \leq d-1$ for $j=1,2,3$. Then, by  \cite[Theorem 2.3]{DSIn} we have $d_1+d_2> d> d-1$ and hence
$$\sigma=3(d-1)-(d_1+d_2+d_3)<3(d-1)-(d-1)-d_3=2(d-1)-d_3=k_3.$$ 
By Theorem \ref{THM10}, the minimal resolution of $N(f)$ is 
$$0\to S(-e)\to\displaystyle{\oplus_{j=1}^3S(-\ell_j)}\to\displaystyle{\oplus_{i=1}^3S(-k_i)}\to S(e-3(d-1)),$$
where $\ell_j=d-1+d_j$. 
We note that $k_3\leq k_2\leq k_1$ and also $k_2 \leq T_0$. If we fix $k$ with $\sigma \leq k<k_3$, the minimal resolution of $N(f)$ above yields
$$n(f)_k=\dim S_{k-\sigma}   = \binom{k-\sigma+2}{2}.$$
Now we consider the case $k_3 \leq k< k_2$. We have  $\ell_1 > k_2$, since $d_1+d_2>d$ as we have seen above. It follows that
$$n(f)_k    
    =\dim S_{k-\sigma}-\dim S_{k-k_3}    = \binom{k-\sigma +2}{2}-\binom{k-k_3+2}{2}.$$
This difference is a linear form in $k$, and the coefficient of $k$ is given by $(k_3-\sigma)$. Note that
$k_3-\sigma=2(d-1)-d_3-3(d-1)+e=d_1+d_2-(d-1)\geq 2$. 
To continue, we need to discuss the position of $\ell_1$ with respect to $k_1$. Note that $\ell_1 >k_1$ if and only if $d_1\geq d/2$.
Hence we have to consider two cases.

\noindent {\bf Case 1: $d_1 \geq d/2$}. In this case, we can compute the value $n(f)_k$  for $k \leq k_1-1$ exactly as above, and we get
$$n(f)_k=\binom{k-\sigma+2}{2}-\binom{k-k_3+2}{2}-\binom{k-k_2+2}{2}.$$ 
Note that in this case $T_0=k_1-1=2(d-1)-d_1-1 > 2d-4+d_1$ and hence all the Hilbert vector $(n(f)_j)$ is known by using Theorem \ref{THMD} (i).

\noindent {\bf Case 2: $d_1 < d/2$}. In this case $\ell_1 \leq k_1$, and we can compute the value $n(f)_k$  for $k \leq \ell_1-1$ exactly as above, obtaining the same formula. Note that in this case $T_0=\ell_1-1=d-2+d_1 > d-3+d_1$, and hence again all the Hilbert vector $(n(f)_j)$ is known by using Theorem \ref{THMD} (ii).

\end{proof}

\begin{exm} \label{ex1}
   Let $C: f=(x^9+y^4z^5)^7+xz^{62}=0$, a singular curve of degree $d=63$. It is a $3-$syzygy curve, not plus-one generated, with $d_1=9$, $d_2=56$ and  $d_3=62$. We have
   $$
   e=\displaystyle{\sum_{i=1}^3d_i=127}, \ \sigma=59, \  k_3=62, \ k_2=68.$$ Since $d_1<\frac{d}{2}$, $T_0=d+d_1-2=70.$ 
   The first quadratic part is for $k \in [59,61]$, the middle linear part is for
   $k \in [62,67]$,  and the second quadratic part is for $k \in [68,70]$. 
   This second quadratic part is too short, containing only 3 points $(j,n(f)_j)$, to be seen in a graphical representation of the corresponding Hilbert vector. Note also that one has $n(f)_k=\nu(C)=27$ for $k\in [69, 91]$, where $91=\lfloor T/2 \rfloor$.
 In particular, the last two points on the second quadratic part are in fact situated on this horizontal line segment. 
     \end{exm}   
 
\begin{exm} \label{ex2}

Let $C: f=(x+y)^2(x-y)^2(x+2y)^2(x-2y)^2(x+3y)^2(x-3y)^2(x+4y)^2(x-4y)^2(x+5y)^2(x-5y)^2+z^{20}=0$, a singular curve of degree $d=20$. It is a $3-$syzygy curve not plus-one generated, with $d_1=9$ and $d_2=d_3=19$. We have
   $$
   e=\displaystyle{\sum_{i=1}^3d_i=47}, \ \sigma=10, \  k_3=k_2=19.$$
   Since $d_1<\frac{d}{2}$, $T_0=d+d_1-2=27=T/2.$
 The first quadratic part is for $k \in [10,18]$, the middle linear part is missing since $k_2=k_3$, the second quadratic part is for $k \in [19,27]$ and one has $n(f)_k=\nu(C)=81$ for $k\in [26, 27]$, where $27=T/2 =T_0$.
     \end{exm}

\begin{figure}[h]
\centering
\begin{tikzpicture}[scale=1.7]
 
 \draw[black] (0,6.5)--(6,6.5);
\draw[black] (0,6.5)-- (0,10.2);
 \draw[dashed] (2,6.7)--(2,6.5);
 \draw[dashed] (2,6.7)--(0,6.7);
 \draw[dashed] (3.8,8.8)--(3.8,6.5);
 \draw[dashed] (3.8,8.8)--(0,8.8);
 \draw[dashed] (5.4,9.9)--(5.4,6.5);
\draw[dashed] (5.4,9.9)--(0,9.9); 

 \node at (2,6.7) {$\bullet$};
 \node at (2,6.5) {$\bullet$};
 \node at (0,6.7) {$\bullet$};
 
 \node at (3.8,8.8) {$\bullet$};
 \node at (3.8,6.5) {$\bullet$}; 
 \node at (0,8.8) {$\bullet$};
  
 \node at (5.4,9.9) {$\bullet$};
  \node at (5.4,6.5) {$\bullet$}; 
 \node at (0,9.9) {$\bullet$};

 
 
 
 
 \draw[black] plot[smooth, domain=2:3.8] (\x, {0.5*((\x)^2-3.4*\x+16.2)} );
 
\draw[black] plot[smooth, domain=3.8:5.4] (\x, {0.5*((\x)^2-3.4*\x+16.2)-((\x)^2-7*\x+17*18/25)});
 
  
 

 \node[right] at (1.9,6.3) {$\sigma$};
  \node[right] at (3.6,6.3) {$k_2=k_3$};
 \node[right] at (5.2,6.3) {$T_0=\frac{T}{2}$};
\node[right] at (-0.3,6.7) {1}; 
\node[right] at (-0.4,8.8) {53}; 
\node[right] at (-1.1,9.9) {$\nu(C)=81$};

   \end{tikzpicture}
   \caption{The Hilbert vector for Example \ref{ex2}}
\label{fig:Ex2}
\end{figure}
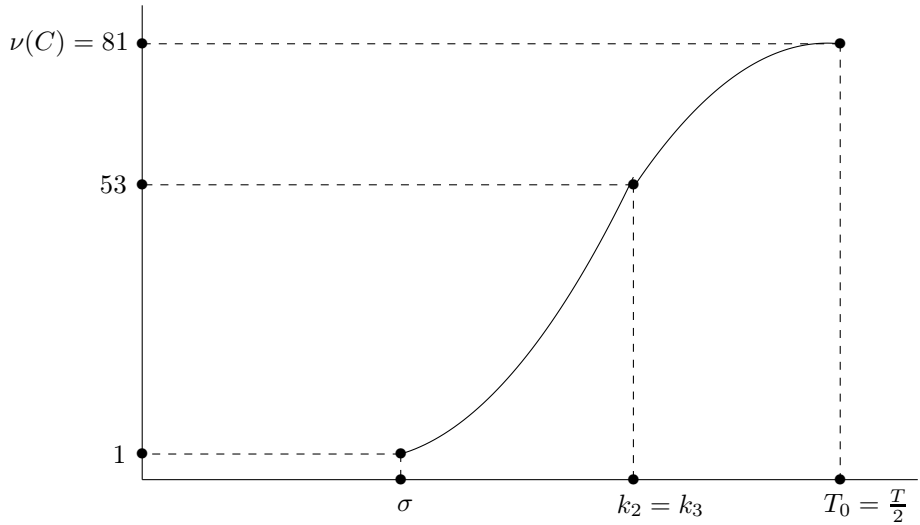


\section{Maximal Tjurina curves and nodal curves}

We assume in this section that $r=d_1 \geq d/2$.\\
A reduced plane curve $C:f=0$ of degree $d$ is called a maximal Tjurina curve if  the global Tjurina number $\tau(C)$ equals the  du Plessis-Wall upper bound, namely if 
\begin{equation}\label{tauM}
\tau(C)=(d-1)(d-r-1)+r^2-{ 2r-d+2 \choose 2},
\end{equation}
see \cite{DStMax,dPW,E}. We know that a reduced plane curve $C:f=0$ of degree $d$ is a maximal Tjurina curve if and only if one has
$d_1=d_2= \dots =d_m=r$, $e_1=e_2=\dots =e_{m-2}=d+r$ and $m=2r-d+3$, see \cite[Theorem 3.1]{DStMax}. 
Using now the equality \eqref{sigma}, it follows that in this case
\begin{equation}\label{sigma2}
\sigma(C)=2d-r-3.
\end{equation}
Theorem \ref{THMD} yields then the following result.

\begin{prop} \label{propMT}
Let $C:f=0$  be  a maximal Tjurina curve  of degree $d$ with $r=d_1 \geq d/2$. Then the Hilbert vector of the Jacobian module $N(f)$ is given by the following
$$n(f)_j=\begin{cases}
     3(d')^2-(j-3d'+2)(j-3d'+1)-\tau(C) \mbox{ for } d=2d'+1\\
     3(d')^2-3d'+1-(j-3d'+3)^2-\tau(C)  \mbox{ for } d=2d'
    \end{cases}$$
    for $2d-3-r\leq j\leq d-3+r$ and $n(f)_j=0$ otherwise.

\end{prop}

Consider now an arbitrary nodal curve $C:f=0$ of degree $d$ in $\p^2$.
Let $\NN$ denote the set of nodes of the curve $C$ and $n(C)$ the number of irreducible components of $C$. For such curves we have the following result.

\begin{thm}\label{thmNodal}
Let $C:f=0$ is a nodal curve in $\p^2$ of degree $d \geq 4$. Then one has the following, with $f_s$ as in Definition \ref{def}.
$$n(f)_k=\begin{cases}
     m(f_s)_k-|\NN| \mbox{ for } d-3 <k \leq T/2 \\
      m(f_s)_k-|\NN|+n(C)-1  \mbox{ for } k=d-3.
      \end{cases}$$
Moreover, when all the irreducible components of $C$ are rational, one has in addition $n(f)_k=0$ for $k \leq d-3$.

\end{thm}

\proof
For any reduced plane curve $C:f=0$, one clearly has
$$n(f)_k= m(f)_k - d(f)_k,$$
where $m(f)_k=\dim M(f)_k$ and $d(f)_k=\dim S_k/(\hat{J_f})_k$.
Since we have to determine $n(f)_k$ only for $k \leq T/2$ by symmetry, and since $ct(f) \geq d-2+r >T/2$ when $r=d_1 \geq d/2$, it follows that
$$n(f)_k= m(f_s)_k - d(f)_k,$$
with $f_s$ as in Definition \ref{def} and  $k \leq T/2$. In particular, for such curves, we have to determine only the values $d(f)_k$ for $k \leq T/2$.
On the other hand, we know that
$$d(f)_k=\tau(C),$$
for $k \geq T-ct(C)$, see \cite[Proposition 2]{DBull}. In particular, this equality holds for $k \geq 3(d-2)-(d-2+r)=2d-4-r$, see also the proof of
\cite[Theorem 3.1]{DSInB}.
Assume now that $C:f=0$ is a nodal curve in $\p^2$. Then $r=d_1 \geq d-2 \geq d/2$ for $d \geq 4$, see \cite[Example 2.2 (i)]{DS14}.
Let  $\defect S_k (\NN)$ denote the defect of the set of nodes $\NN$ with respect to the linear system $S_k$.
Then it is known that
$$d(f)_k=|\NN| - \defect S_k(\NN),$$
see \cite{DBull}.
On the other hand, \cite[Corollary 1.6]{DStEdin} implies that $\defect S_k(\NN)=0$ for $k >d-3$ and $\defect S_k(\NN)=n(C)-1$ for $k=d-3$. 
If all the irreducible components of $C$ are rational, then \cite[Theorem 2.7]{DStRLM} shows that $n(f)_k=0$ for $k \leq d-3$. These facts imply our claims.
\endproof

\section{Relation to a result by Hartshorne}

\noindent Let $C:f=0$ be a curve of degree $d$ in $\p^2$, and let $r=mdr(f)$ be the minimal degree of a Jacobian syzygy for $f$. In this section we  give some informations about the invariant $\sigma(C)$, using a result by Hartshorne, namely \cite[Theorem 7.4]{Hart}.
 We recall that the sheafification of $\Syz(J_f)$, denoted by $E_C:=\widetilde{\Syz(J_f)}$, is a rank two vector bundle on  $\p^2$, see \cite{AD18, S80, Se14}. We set $$e(f)_m=\dim \Syz(J_f))_m=\dim H^0(\p^2, E_C(m)),$$ for any integer $m$. Associated to the vector bundle $E_C$ there is the normalized vector bundle $\E_C$, which is the twist of $E_C$ such that $c_1(\E_C)\in\{-1,0\}$. More precisely,
\begin{description}
\item[when $d=2d'+1$ is odd] 
\begin{equation}\label{cond3}
    \E_C=E_C(d'), \quad c_1(\E_C)=0, \quad c_2(\E_C)=3(d')^2-\tau(C),
\end{equation}
and 
\item[when $d=2d'$ is even] 
\begin{equation}\label{cond4}
    \E_C=E_C(d'-1), \quad c_1(\E_C)=-1, \quad c_2(\E_C)= 3(d')^2-3d'+1-\tau(C),
\end{equation}
\end{description}
see \cite[Section 2]{DSInB}.
\begin{rem}  \label{rk4.1} \rm
The vector bundle $E_C$ is stable if and only if $\E_C$ has no sections, see \cite[Lemma 1.2.5]{OSS80}. This is equivalent to $r=mdr(f)\geq \frac{d}{2},$ see \cite[Proposition 2.4]{Se14}. Moreover by \cite[Theorem 2.2]{DSInB} and using the formulas (\ref{cond3}) and (\ref{cond4}), we have that for a stable vector $E_C$, $c_2(\E_C)=\nu(C).$ 
Moreover, the vector bundle $E_C$ is semistable if and only if $r=mdr(f) \geq (d-1)/2$, see again \cite[Lemma 1.2.5]{OSS80}, a condition that occurs in our Theorem \ref{thmH1} below.
\end{rem}
The important key point is the identification $$H^1(C,E_C(k))=N(f)_{k+d-1}$$ for any integer $k$, see \cite[Proposition 2.1]{Se14}. Hence the study of the Hilbert vector of the Jacobian module $N(f)$ is equivalent to the study of the dimension of $H^1(C, E_C(k))$.

\begin{thm} \label{thmH1}
Let $C:f=0$ be a curve of degree $d$, and let $r=mdr(f)$ be the minimal degree of a Jacobian syzygy for $f$. Assume that $r \geq (d-1)/2$, in other words that the rank 2 vector bundle $E_C$ is semistable. Then we have the following.
 \begin{enumerate}
     \item If $d=2d'+1$ is odd, then
     $$\sigma(C) \geq \tau(C)-2(d')^2-2rd'+r^2+3d'-1.$$
     
     \item If $d=2d'$ is even, then
     $$\sigma(C) \geq \tau(C)-2(d')^2-2rd'+r^2+5d'+r-3.$$
     \end{enumerate}
     The above inequalities are sharp, in particular they are equalities when $C$ is a maximal Tjurina curve with $r\geq d/2$.

\end{thm}

\begin{proof}
We discuss only the case $d=2d'+1$, the other case being completely similar. One has
$$n(f)_k=h^1(\p^2,\E_C(k-3d')).$$
Moreover $h^0(\p^2,\E_C(t))=h^0(\p^2,E_C(t+d')) \ne 0$ if and only if 
$t+d' \geq r$. Hence the minimal $t$ satisfying this condition is $t_m=r-d' \geq 0$. Then \cite[Theorem 7.4]{Hart} implies that
$n(f)_k=0$ when 
$$k-3d' \leq -c_2(\E)+t_m^2-2.$$
Using the formula for $t_m$ above, and the formula for $c_2(\E)$ given in the equations \eqref{cond3}, we get that $n(f)_k=0$ when 
$$k \leq \tau(C)-2(d')^2-2rd'+r^2+3d'-2,$$
which clearly implies our claim (1). The fact that the inequality in (1) is in fact an equality when $C$ is a maximal Tjurina curve with $r\geq d/2$ follows by a direct computation. Indeed, using the above definition of a maximal Tjurina curve of degree $d=2d'+1$, namely the equality \eqref{tauM}, we see that
$$\tau(C)=2(d')^2+2rd'-r^2-r+d'.$$
Hence
$$\tau(C)-2(d')^2-2rd'+r^2+3d'-1=2d-r-3=\sigma(C),$$
where the last equality follows from \eqref{sigma2}.
\end{proof}

\begin{exm} \label{uninode}
Let $C:f=0$ be a curve of degree $d=2d'+1$, having a unique node as singularities. Then it is known that $r=d-1=2d'$, and $\tau(C)=\sigma(C)=1$. The inequality in Theorem \ref{thmH1} (1) is in this case
$$1 \geq d'(3-2d'),$$
hence the two terms in this inequality can be far apart in some cases.
\end{exm}

\end{document}